\documentclass[12pt]{amsart}

\usepackage{fullpage}
\usepackage{amsmath}
\usepackage{amsfonts}
\usepackage{amssymb}
\usepackage{amsthm}
\usepackage[all,knot,poly]{xy}

\input xy
\xyoption{all}

\theoremstyle{plain}

\newtheorem{theorem}{Theorem}%[section]
\newtheorem{lemma}[theorem]{Lemma}
\newtheorem{proposition}[theorem]{Proposition}
\newtheorem{corollary}[theorem]{Corollary}
\newtheorem{remark}[theorem]{Remark}
\newtheorem{definition}[theorem]{Definition}
\newtheorem{example}[theorem]{Example}

\usepackage{color}

\def\blue{\textcolor{blue}}

\input xy
\xyoption{all}

\title{\blue{Motion planning algorithms for Configuration Spaces}}
\author{Hugo Mas-Ku, Enrique Torres-Giese}

\begin{document}
\begin{abstract} We provide explicit motion planners for Euclidean configuration spaces. 
This allows us to recover some known values of the topological complexity and the Lusternik-Schinirelman category of these spaces.  
\end{abstract}
\maketitle

\section{Introduction}

The Topological Complexity (TC) of a space $X$ is, in practical terms, the smallest number of local domains in each of which there is a 
continuous motion planning algorithm. This number turns out to be a homotopy invariant and is denoted by $TC(X)$, see~\cite{farber}. 
Recall that the space of configurations of $k$ labeled points in $X$ is given by
\[ F(X,k) = \{ (x_1,\ldots,x_k)\in X^k: x_i \neq x_j \} .\]

This space turns out to be of crucial importance in algebraic topology and of course in many of its applications in fields such as robotics.
This latter connection arises by noticing that a path in $F(\mathbb{R}^n,k)$ is essentially a set of $k$ non-colliding paths in $\mathbb{R}^n$.

A related concept is that of the orbit configuration space of a $G$-space $X$, which is defined as
\[ F_G(X,k) = \{ (x_1,\ldots,x_k)\in X^k: Gx_i\neq Gx_j \} ,\]
where $Gx= \{ gx: g\in G \}$ is the orbit $x$. For instance, the space $F_{O(n)}(\mathbb{R}^n,k)$, where $O(n)$ is the linear orthogonal group, is the subspace of $F(\mathbb{R}^n,k)$ consisting of configurations whose components are vectors of different lengths. Note that $F_1(X,k)$, where 1 is the trivial group, is just the space of configurations of $k$ labeled points in $X$.

We will construct explicit motion planners on $F(\mathbb{R}^n,k)$ that realize the value of its TC when $k$ is odd, and when $k$ is even this
number of motion planners is just one unit off the actual value of its TC. Before embarking into the construction of these planners we will 
provide a lower bound for their TC by constructing a retract of $F(\mathbb{R}^n,k)$ that realizes the value of TC of $F(\mathbb{R}^n,k)$ when $n$ is
odd, as well as the TC of some orbit configuration spaces.

The values of TC and the Lusternik-Schnirelmann category ($cat$) of $F(\mathbb{R}^n,k)$, had already been computed (\cite{fy},~\cite{fg},~\cite{roth}), and for $n,k\geq 2$ are given by
\[
TC(F(\mathbb{R}^n,k) )=\left\{ \begin{array}{cc}
2k -1 & n \mbox{ odd} \\
2k- 2 & n \mbox{ even} \end{array} \right. ,
\]
and
\[ cat(F(\mathbb{R}^n,k)) =k.\]
The conditions $n,k\geq 2$ guarantee the space $F(\mathbb{R}^n,k)$ is a non-contractible, connected space.
The contribution of this paper is two-fold, on one hand we show that no sophisticated machinery is needed to compute $cat$ nor TC when $n$ is odd; and on the other 
we find explicit motion planning algorithms for $F(\mathbb{R}^n, k)$ with $2k-1$ local rules, solving a problem posed in~\cite{farber}. Previous motion planning algorithms described 
in~\cite{farber2} consisted of $k^2-k+1$ local rules.

\section{Retracts for Configuration Spaces}

In this section we will show that there exist retracts given by products of spheres sitting in the configuration spaces that we will consider. 
This will allow us to obtain lower bounds for $cat$ and TC.

\begin{proposition}\label{retraction} If $G$ is a subgroup of $O(m+1)$, then there are retractions
\[  (S^m)^{k-1}\hookrightarrow F_G(\mathbb{R}^{m+1}, k) \to (S^m)^{k-1} \]
and
\[ (S^m)^k\hookrightarrow F_G(\mathbb{R}^{m+1}\setminus \{ 0 \} , k) \to (S^m)^k .\]
\end{proposition}
\begin{proof}  To see these let
\[ \alpha_1: (S^m)^{k-1} \to F_G(\mathbb{R}^{m+1} , k) \]
\[ (x_1,\ldots,x_{k-1}) \mapsto (0,x_1,x_1+3x_2,\ldots, x_1+3x_2+\cdots+3^{k-2}x_{k-1} ) \]
and
\[ \beta_1: F_G(\mathbb{R}^{m+1} , k) \to (S^m)^{k-1} \]
\[ (y_1,\ldots,y_k) \mapsto (N(y_2- y_1),\ldots, N(y_k - y_{k-1})) ,\]
where $N(y) = y/|y|$.  Similarly, let us define
\[ \alpha_2: (S^m)^k \to F_G(\mathbb{R}^{m+1}\setminus \{ 0 \} , k) \]
\[ (x_1,\ldots,x_k) \mapsto (x_1,x_1+3x_2,\ldots, x_1+3x_2+\cdots+3^{k-1}x_k ) \]
and
\[ \beta_2: F_G(\mathbb{R}^{m+1}\setminus \{ 0 \} , k) \to (S^m)^k \]
\[ (y_1,\ldots,y_k) \to (N(y_1),N(y_2- y_1),\ldots, N(y_k - y_{k-1})) \]

These maps satisfy $\beta_i\circ\alpha_i = 1$, for $i=1,2$.
By definition $\beta_i$ lands in the respective product of spheres. One only needs to
show that the map
\[ \alpha_2(x_1,\ldots, x_k) = (x_1, x_1 + 3x_2,\ldots, x_1 + 3x_2 + \dots + 3^{k-1}x_k )\]
does land in $ F_G(\mathbb{R}^{m+1}\setminus \{ 0\} , k)$. The case of $\alpha_1$ is analogous.
It is easy to see that each coordinate of this map is non-zero. If $A(x_1+\cdots+3^{l-1}x_l) = x_1+\cdots+ 3^{l+p-1} x_{l+p}$ for some $A\in G\subseteq O(m+1)$ and $p\geq 1$, then 
$|x_1+\cdots+3^{l-1}x_l| = |x_1+\cdots+ 3^{l+p-1} x_{l+p}|$. Now, if any two vectors
$u$ and $v$ satisfy $|u|= |u+v|$, then $|v| \leq 2|u|$. Thus, if we take $u=x_1+\cdots+3^{l-1}x_l$ and
$v=  3^lx_{l+1}+\cdots+ 3^{l+p-1} x_{l+p}$, then
\[ 3^l| x_{l+1}+\cdots+ 3^{p-1} x_{l+p}| \leq 2|x_1+\cdots+3^{l-1}x_l|<3^l, \]
So $| x_{l+1}+\cdots+ 3^{p-1} x_{l+p}|< 1$. On the other hand,
\[ 1\leq \frac{3^{p-1} +1}{2} \leq |3^{p-1} - |x_{l+1}+\cdots+ 3^{p-2} x_{l+p-1}||\leq |x_{l+1}+\cdots+ 3^{p-1} x_{l+p}| ,\]
a contradiction. Therefore the vector $(x_1, x_1 + 3x_2,\ldots, x_1 + 3x_2 + \dots + 3^{k-1}x_k )$ does live in the configuration space
$F_G(\mathbb{R}^{m+1} \setminus \{ 0 \}, k)$.  \end{proof}

\begin{remark} The arguments in the proof of the previous result can be used to show that there is also a retraction
\[ (S^m)^k\hookrightarrow F(\mathbb{R}^{m+1} \setminus Q_r, k) \to (S^m)^k\]
where $Q_r$ is a subset of fixed points with $r$ elements. This follows from the homeomorphism induced between configuration
spaces by a homeomorphism between $\mathbb{R}^{m+1} \setminus Q_r$ and $\mathbb{R}^{m+1} \setminus \overline{Q}_r$, where
$\overline{Q}_r$ is a subset of $r$ fixed points of norm less than one (note that each component of the map $\alpha_2$ is a vector of norm greater than or equal to 1).
The space $F(\mathbb{R}^{m+1} \setminus Q_r, k)$ is related to the  collision free motion planning problem in the presence of multiple moving obstacles, see~\cite{fgy}. 
\end{remark}

\begin{theorem}\label{tc-inequalities} Suppose that $Q_r$ is a set of $r$ points in $\mathbb{R}^{m+1}$, then
\begin{itemize}
\item[(1)] \[ cat(F(\mathbb{R}^{m+1}\setminus Q_r, k)) =\left\{\begin{array}{cl}
 k & \mbox{if } r=0\\
k+1 & \mbox{if } r>0\end{array}\right.\]
\item[(2)] \[ TC(F(\mathbb{R}^{2m+1}\setminus Q_r, k)) =\left\{\begin{array}{cl}
 2k-1 & \mbox{if } r=0\\
2k+1 & \mbox{if } r>0\end{array}\right.\]
\end{itemize}
Suppose that $G$ is a finite subgroup of $O(2m+1)$ acting freely on $\mathbb{R}^{2m+1}\setminus \{ 0 \}$, then
\begin{itemize}
\item[(3)] \[ cat(F_G(\mathbb{R}^{m+1}\setminus \{ 0\}, k)) =k+1\]
\item[(4)] \[TC( F_G(\mathbb{R}^{2m+1} \setminus \{ 0 \}, k)) = 2k +1.\]
\end{itemize}
\end{theorem}
\begin{proof}
Notice that there are fibrations of the form
\[ \bigvee^{r+k-1} S^m \to F(\mathbb{R}^{m+1}\setminus Q_r,k) \to F(\mathbb{R}^{m+1}\setminus Q_r,k-1) \]
and
\[ \bigvee^{g(k-1)+1} S^m \to F_G(\mathbb{R}^{m+1}\setminus \{0\},k) \to F_G(\mathbb{R}^{m+1}\setminus \{0\},k-1) \]
where $g$ is the order of $G$. 
An inductive argument shows that the space $F_G(\mathbb{R}^{m+1}\setminus \{0\},k)$ is $(m-1)$-connected and homotopy equivalent 
to a finite CW-complex of dimension at most $mk$.  Similarly, the space $F(\mathbb{R}^{m+1}\setminus Q_r,k)$ is $(m-1)$-connected and
homotopy equivalent to a CW complex of dimension at most $m(k-1)$ when $r=0$, and of dimension at most $mk$ when $r>0$. To get a lower bound for
$cat$ and TC of these spaces we just need  to apply Proposition~\ref{retraction}, recall the fact that if $X$ is dominated by $Y$ then $TC(X)\leq TC(Y)$,
and also make use of the known values $TC((S^{2m})^k)=2k+1$ and $cat((S^m)^k) = k+1$.
For the upper bounds we can apply the following two properties: $TC(X)\leq 2cat(X)-1$; and if $X$ is a $q$-connected finite CW-complex then 
$cat(X)\leq \frac{\dim(X)}{q+1} +1$.
\end{proof}

\begin{remark}
The value of $TC(F(\mathbb{R}^{2}\setminus Q_r, k))$ and $TC(F(\mathbb{R}^{3}\setminus Q_r, k))$ had already been computed in~\cite{fgy}, and more recently for any 
Euclidean space in~\cite{gg}.
\end{remark} 

\section{Partitions on Configuration Spaces}

Throughout this section we will be working with the space $F(\mathbb{R}^2,k)$, and we will keep
$k$ fixed. A vector of positive integers $A=(a_1,\ldots,a_l)$ such that $\sum a_i =k$ will be
called a partition of $k$, and we will call the number $|A| = l$ the number of levels of $A$. We will consider the (reverse) lexicographic order on $\mathbb{R}^2$, that
is: $(b_1,b_2)\leq (c_1,c_2)$ if $b_2<c_2$, or if $b_2=c_2$ and $b_1\leq c_1$.

Now, if $x=(x_1,\ldots,x_k)\in F(\mathbb{R}^2,k)$ then there is a unique permutation $\sigma \in \Sigma_k$ such
that $x_{\sigma(1)} <\cdots < x_{\sigma(k)}$. This permutation will be denoted by $\sigma_x$, and if 
$\sigma_x =1$ we will say that $x$ is (lexicographically) ordered.

Let $\pi_2:\mathbb{R}^2\to \mathbb{R}$ be the projection of the second factor. If 
 $x=(x_1,\ldots,x_k)\in F(\mathbb{R}^2,k)$ is (lexicographically) ordered, then there are positive integers $a_1,\ldots,a_l$ such that
\begin{align*}
 \pi_2(x_1)= \cdots =\pi_2(x_{a_1}) & <   \pi_2( x_{a_1+1}) \\
\pi_2(x_{a_1+1})= \cdots =\pi_2(x_{a_1+a_2}) & < \pi_2(x_{a_1+a_2+1}) \\
  & \vdots  \\
\pi_2(x_{a_1+...+a_{l-2}+1})= \cdots =\pi_2(x_{a_1+...+a_{l-1}}) & <   \pi_2(x_{a_1+...+a_{l-1}+1}) \\
 \pi_2(x_{a_1+...+a_{l-1}+1})= \cdots =\pi_2(x_{a_1+...+a_l}) 
\end{align*}

These of course define a partition $(a_1,\ldots,a_l)$ of $k$. This partition will be denoted by $A_x$. 
Note that this partition tells us how the configuration $x$ is sitting in $\mathbb{R}^2$ with respect to
the $y$-axis. In this context, $|A|=l$ is the number of lines parallel to the $x$-axis on which 
the configuration $x$ sits.

\begin{definition} Given a partition $A=(a_1,\ldots,a_l)$ of $k$ and 
$x=(x_1,\ldots,x_k)\in F(\mathbb{R}^2,k)$, we will say that $x$ is an $A$-configuration if 
$A_{\sigma_x(x)} = A$.
\end{definition}

\begin{definition} Given an $A$-configuration 
$x=(x_1,\ldots,x_k)\in F(\mathbb{R}^2,k)$, we will say that $x$ has $|A|$ levels and that
$x_i$ and $x_j$ are on the same level if $\pi_2(x_i)=\pi_2(x_j)$.
\end{definition}

\begin{definition} Given a partition $A$ of $k$ and a permutation $\sigma \in \Sigma_k$, we let
\[ F_{A,\sigma} =\{ x=(x_1,\ldots,x_k)\in F(\mathbb{R}^2,k): \sigma_x = \sigma 
\mbox{  and } x \mbox{ is an } A\mbox{-configuration} \}. \] We also define 
\[ F_A = \bigcup_{\sigma\in \Sigma_k} F_{A,\sigma} \]
\end{definition}
This latter is precisely the subspace of all $A$-configurations.
Note that the subspaces $F_{A,\sigma}$ are disjoint, and that  
\[ F(\mathbb{R}^2,k) = \bigcup_A F_A . \]

\begin{theorem} Suppose that $x \in F(\mathbb{R}^2,k)$ is a limit point of $F_{A,1}$, then $|A_x| \leq |A|$. The equality holds if and only if, 
$A_x = A$.
\end{theorem}
\begin{proof}
Suppose that $|A|=l$. Note that any element of a sequence of (lexicographically) ordered 
$A$-configurations converging to $x$ defines a set of increasing real numbers $h_1<\cdots< h_l$ which are determined by the map $\pi_2$.
Moreover, this latter set of real numbers depends continuously on the sequence converging to $x$. The position of the levels of $x$ is
determined by the limit of these real numbers, and since some of these may collapse into a single real number in the limit, it follows that $|A_x| \leq |A|$. 

For the second part, it suffices to show that if $|A_x| = |A|$ then $A_x = A$. This can be seen by
noticing that the condition $|A_x| = |A|$ tells us that the levels determined by the sequence do 
not collapse resulting in a smaller number of levels when converging to $x$, and hence $x$ must be an $A$-configuration.
\end{proof}

Note that since the subspaces $F_{A,\sigma}$ and $F_{A,\mu}$ are homeomorphic for any two
permutations $\sigma,\mu$, it follows that the latter result holds for any $F_{A,\sigma}$.

\begin{corollary}\label{limit_point} Suppose that $(x,y) \in F(\mathbb{R}^2,k) \times F(\mathbb{R}^2,k)$ is a limit point of $F_{A,\sigma}\times F_{B,\mu}$. Then $|A_x| + |A_y| \leq |A| + |B|$, and the equality holds if
and only if $A_x = A$ and $B_y = B$.
\end{corollary}
\begin{proof}
The result follows from the following observation: if $a,b,c,d$ are postive real numbers such that 
$a\geq c$, $b\geq d$, then $a+b \geq c+d$, and the equality holds if and only if $a=c$ and
$b = d$.
\end{proof}

Recall that a space $X$ is called ENR (Euclidean Neighborhood Retract) if it is homeomorphic to a subspace $X'$ of some 
$\mathbb{R}^N$ such that $X'$ is a retract of an open neighborhood $X'\subset U\subset \mathbb{R}^N$. Here we recall a 
definition of TC from~\cite{farber}, Proposition 4.12.

\begin{definition} Suppose that $X$ is an ENR. The topological complexity of $X$ is the smallest integer $r$ such that there 
exists a section $s:X\times X \to X^I$ of the double-evaluation map $\epsilon:X^I\to X\times X$ and
a splitting $F_1\cup\cdots\cup F_r = X\times X$ such that: 
\begin{enumerate}
\item $F_i\cap F_j =\emptyset$ when $i\neq j$,
\item the restriction of $s$ to each $F_i$ is continuous, and
\item each $F_i$ is a locally compact subspace of $X\times X$.
\end{enumerate}
\end{definition}

\begin{definition} Given $i\in \{ 2,...,2k\}$, we let
\[ F_i = \bigcup_{|A|+|B| = i} F_A \times F_B .\]
\end{definition}

Note that these $F_i$ are disjoint and they cover $F(\mathbb{R}^2,k)\times F(\mathbb{R}^2,k)$.

\begin{example} When $k=3$, the first two $F_i$ are given as follows
\begin{align*}
 F_2 & = F_{(3)} \times F_{(3)},  \\
 F_3 &= F_{(3)} \times F_{(1,2)} \cup F_{(3)} \times F_{(2,1)} \cup F_{(1,2)} \times F_{(3)} \cup F_{(2,1)} \times F_{(3)}.
% F_4 &=  F_{(1,2)} \times F_{(1,2)} \cup F_{(1,2)} \times F_{(2,1)} \cup F_{(2,1)} \times F_{(1,2)}
%\cup F_{(2,1)} \times F_{(2,1)} \cup F_{(3)} \times F_{(1,1,1)} \cup F_{(1,1,1)} \times F_{(3)} .
\end{align*}
\end{example}

\begin{lemma}\label{enr}  Each $F_{A,\sigma}$, $F_A$ and $F_i$ are locally compact, locally contractible, and hence ENR.
\end{lemma}
\begin{proof}
If $A=(a_1,\ldots,a_l)$ is a partition of $k$, then there is a homeomorphism
\[ F_{A,1} \to F(\mathbb{R},a_1)\times\cdots\times F(\mathbb{R},a_l)\times \tilde{F}(\mathbb{R},l) ,\]
where $\tilde{F}(\mathbb{R},l) = \{ (h_1,\ldots,h_l)\in \mathbb{R}^l: h_1 < \cdots< h_l \}$. 
This homeomorphism is obtained by projecting each level onto the $x$-axis, and by
projecting each level onto the $y$-axis. Therefore $F_{A,1}$ is homeomorphic to an open set
of $\mathbb{R}^{k+l}$, and thus each $F_{A,\sigma}$,  $F_A$, and $F_i$ are locally compact, and locally contractible.
Finally, a subspace of $\mathbb{R}^N$ is an ENR if and only if, it is locally compact and locally contractible \cite{dold}.
\end{proof}

\begin{lemma}  If $V=F_{A,\sigma} \times F_{B,\mu}\subset F_i$ and $(x,y) \in \overline{V} - V$, then $(x,y)\in F_j$ for some $j<i$. 
\end{lemma}
\begin{proof}
Note that if $x$ is a limit point of $F_{A,\sigma}$ and $|A_x|=|A|$, then $x\in F_{A,\sigma}$. Now apply Corollary~\ref{limit_point}.
\end{proof}

\begin{lemma} Suppose that $U$ and $V$ are disjoint subspaces of $\mathbb{R}^N$ such that $\overline{U}\cap V$ and 
$U\cap \overline{V}$ are empty. If $f$ is a function from $\mathbb{R}^N$ to $\mathbb{R}^M$ such that $f$ restricted to both $U$ and $V$ is continuous, then $f$ is continuous on $U\cup V$.
\end{lemma}

These latter two results are crucial since they tell us that if we are able to find a planner on $F_i$ then 
it will be continuous on $F_i$ as long as it is continuous on each 
$F_{A_1,\sigma_1}\times F_{A_2,\sigma_2}\subset F_i$.

\section{Motion Planners}

The following result will be a basic ingredient needed to construct motion planners and its proof will be omitted since it is straightforward.

\begin{lemma}\label{distance} Let $\pi_1:\mathbb{R}^2\to \mathbb{R}$ be the projection of the first factor, and define 
$p:(\mathbb{R}^2)^{2k} \to \mathbb{R}$ by $(x_1,\ldots,x_{2k})\mapsto \max_{1\leq j\leq 2k} \{ \pi_1(x_j)\}$.
The map $p$ is continuous, and so is its restriction to $F(\mathbb{R}^2,k)\times F(\mathbb{R}^2,k)$.
\end{lemma}

We will define a planner $s_i$ on $F_i$ by means of planners $s_{A,\sigma,B,\mu}$ on each $F_{A,\sigma}\times F_{B,\mu} \subset F_i$, where $i=|A|+|B|$.
Without loss of generality we will provide a recipe only for $F_{A,1}\times F_{B,1}\subset F_i$: 
\begin{enumerate}
\item Take a pair of configurations $(x,y) \in F_{A,1}\times F_{B,1}\subset F_i$.
\item Each level of the $A$-configuration $x$ will be connected by means of straight lines to a set of points 
on a line which is parallel to the $y$-axis and whose $x$-coordinate is given by $p(x,y)+1$. More precisely, if $A_x = (a_1,\ldots,a_l)$
and we let $h_j = \pi_2(x_{a_1+\cdots+a_j})$, then 
\begin{enumerate} 
\item $x_1,\ldots,x_{a_1}$ will be mapped onto the line $X=p(x,y)+1$ by means of straight lines,
$x_1$ will go to the point on the line at height $h_1 - |x_1 - x_{a_1}|$, $x_2$ will go to the point on the
line at height $h_1 - |x_2 -x_{a_1}|$, and so on.
\item For the next level, send $x_{a_1+j}$ to the point on the line $X=p(x,y)+1$ at height 
\[ h_2 - \frac{( a_2 - j  )(h_2 - h_1)}{2(a_2 - 1)}, \] for $1\leq j \leq a_2$.
\item Proceed as in (b) with each level of $x$.
\end{enumerate}
This set of paths define a path $Q_x$ in $F(\mathbb{R}^2,k)$ connecting the configuration $x$ to a configuration sitting on the
line $X=p(x,y)+1$.

\item We proceed with $y$ the same way we did with $x$ to obtain a path $Q_y$ except that in this case we use the line $X=p(x,y)+2$ to
avoid possible collisions in the following step.
\item Let $\alpha_{(x,y)}$ be the path that connects by means of straight lines (following the order of both $x$ and $y$) the configuration $Q_{x}(1)$ to the configuration $Q_y(1)$. 
\item The motion planner is determined by the path from $x$ to $y$ given by $Q_x\cdot\alpha_{(x,y)} \cdot Q_y^{-1}$ (concatenation of paths).
 \end{enumerate}

The following picture illustrates the construction of the path $Q_x$ when $A=(3,2,1,2)$.
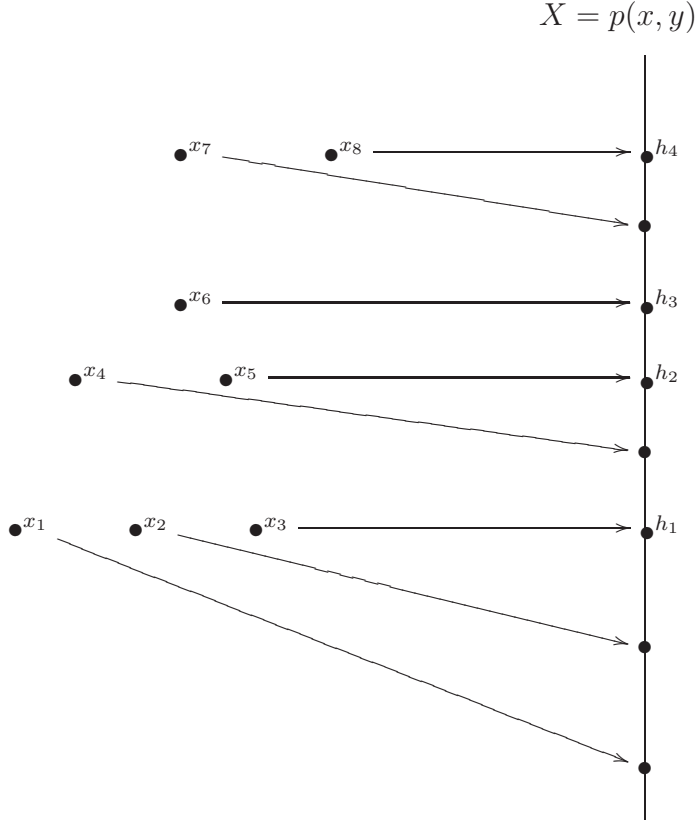
\[
\xygraph{
!{<0cm,0cm>;<2cm,0cm>:<0cm,2cm>::}
!{(.4,0) }*+{\bullet^{x_1}} ="x1"
!{(1.2,0) }*+{\bullet^{x_2}} ="x2"
!{(2,0) }*+{\bullet^{x_3}} ="x3"
!{(.8,1) }*+{\bullet^{x_4}} ="x4"
!{(1.8,1) }*+{\bullet^{x_5}} ="x5"
!{(1.5,1.5) }*+{\bullet^{x_6}} ="x6"
!{(1.5,2.5) }*+{\bullet^{x_7}} ="x7"
!{(2.5,2.5) }*+{\bullet^{x_8}} ="x8"
!{(4.6,0) }*+{\bullet^{h_1}} ="h1"
!{(4.5,-.8) }*+{\bullet} ="h12"
!{(4.5,-1.6) }*+{\bullet} ="h13"
!{(4.6,1) }*+{\bullet^{h_2}} ="h2"
!{(4.5,.5) }*+{\bullet} ="h22"
!{(4.6,1.5) }*+{\bullet^{h_3}} ="h3"
!{(4.6,2.5) }*+{\bullet^{h_4}} ="h4"
!{(4.5,2.0) }*+{\bullet} ="h42"
!{(4.5,-2) }*+{} ="A"
!{(4.5,3.4) }*+{X=p(x,y)+1} 
!{(4.5,3.2) }*+{} ="B"
"A"-"B"
"x1":"h13" "x2":"h12"  "x3":"h1"
"x4":"h22" "x5":"h2"
"x6":"h3"
"x7":"h42" "x8":"h4"
}
\]

\begin{theorem} The collection $(F_i,s_i)$, $2\leq i\leq 2k$, forms a set of motion planning algorithms for $F(\mathbb{R}^2,k)$.
\end{theorem}

\section{Higher dimensions and higher TC}

For simplicity and convenience we will denote the coordinates of $\mathbb{R}^n$ by $z_1,\ldots,z_n$. 
We can extend the ideas of partitions and levels to this scenario: given a configuration $x\in F(\mathbb{R}^n,k)$, each level will be a hyperplane perpendicular to 
the $z_n$--axis  containing a number of elements of $x$, and this number of elements is a component of the partition determined by $x$. 
Now, given $y\in F(\mathbb{R}^n,k)$, we define 
\[ p(x,y) = \max_{1\leq i,j \leq k} \{ \pi_{1}(x_i), \pi_{1}(y_j) \} ,\]
where $\pi_{1}$ is the projection of the first factor, see Lemma \ref{distance}.
Then the elements on a level of $x$ are connected to a configuration on the line $L_{x,y}$ which is parallel to the $z_n$-axis and intersects the $z_1$-axis at $p(x,y)+1$. The recipe spelled out for $\mathbb{R}^2$ works for $\mathbb{R}^n$, the only difference is that we will 
consider the lexicographic order on each level to assign to each point a point on the line $L_{x,y}$ (this is implicit in step (2)(b) for $\mathbb{R}^2$). \\

The concept of higher topological complexity was developed in~\cite{bgrt}. The basic idea is that in this case the motion planning involves a set of $(n-2)$ prescribed
intermediate stages that the system (robot) has to reach. This turns out to be an invariant and it is denoted by $TC_n$. The case $n=2$ is just that of $TC$. The arguments
applied in the proof of Theorem~\ref{tc-inequalities} can be used in this context since the analogous ideas for $TC_n$ are available in~\cite{bgrt}. This allows us to obtain 
 \[ TC_n(F(\mathbb{R}^{2m+1}\setminus Q_r, k)) =\left\{\begin{array}{cl}
 n(k-1)+1 & \mbox{if } r=0\\
nk+1 & \mbox{if } r>0\end{array}\right. ,\]
and if $G$ is a finite subgroup of $O(2m+1)$ acting freely on $\mathbb{R}^{2m+1}\setminus \{ 0 \}$, then
 \[TC_n( F_G(\mathbb{R}^{2m+1} \setminus \{ 0 \}, k)) = nk+1.\]

The value of $TC_n(F(\mathbb{R}^m\setminus Q_r, k))$ was obtained in~\cite{gg}; their arguments, however, are way more elaborate. \\

It is also worth mentioning that the motion planning algorithms described in this paper can also be extended to the case of higher topological complexity. It is not hard to 
see what modifications are needed, and the details are left to the interested reader. 

As we pinpointed in the introduction, the contribution of this paper resides more in the construction of the motion planners. This construction may be of more practical importance than just knowing the value of TC.

\section{LS-category}

The LS-category of $F(\mathbb{R}^n,k)$ has been computed in \cite{roth} and it is equal to $k$ when $n\geq 2$. We will construct a categorical 
cover that realizes this value. Consider the sets 
\[ W_i = \bigcup_{|A|=i} F_A,\]
where $i\in\{ 1,\ldots,k\}$
and notice that they are ENR by Lemma \ref{enr}. Now we use the following result from \cite{dold}.

\begin{lemma} If $W$ is a subspace of $X$ and both are ENR, then there is an open neighborhood $W\subset U \subset X$ and a retraction 
$r : U \to W$ such that the natural inclusion $ j : U \to X$ is homotopic to $i\circ r$, where $i$ is the natural inclusion map of $W$ into $X$.
\end{lemma}

\begin{theorem} The subspaces $W_i$, $1\leq i\leq k$, can be enlarged  to define a categorical covering for $F(\mathbb{R}^n,k)$.
\end{theorem}
\begin{proof}
Note that each $W_i$ is contractible in $F(\mathbb{R}^n,k)$ by using the ideas from steps (1) and (2) in the defintion of the motion planners and 
by connecting the resulting configurations on the corresponding line (see step (2)(a)) to a fixed configuration in $\mathbb{R}^n$. A straightforward application of the previous result allows us to enlarge each subspace $W_i$ to an open subset $U_i\subset F(\mathbb{R}^n,k)$ 
so that $U_i$ is contractible in $F(\mathbb{R}^n,k)$.
The fact that $k$ is the smallest possible size of a categorical covering is a consequence of Proposition~\ref{retraction}.
Therefore the subsets $U_1,\ldots, U_k$ define a categorical cover of $F(\mathbb{R}^n,k)$.
\end{proof}


\begin{thebibliography}{CJS}

\bibitem[BGRT]{bgrt} Basabe, I. Gonzalez, J. Rudyak, Y.  Tamaki, D. Higher topological complexity and its symmetrization. To appear in
Algebraic \& Geometric Topology.


\bibitem[Dold]{dold} Dold, A. Lectures on Algebraic Topology, Springer–Verlag, 1972.

\bibitem[Farber 08]{farber} Farber, M. Invitation to Topological Robotics. Zurich Lectures in Advanced Mathematics. European Mathematical Society (EMS), Z\"urich, 2008.

\bibitem[Farber 06]{farber2} Farber, M. Topology of robot motion planning, In "Morse Theoretic Methods in Nonlinear Analysis and in Symplectic Topology”, Paul Biran, Octav Cornea, Francois Lalonde editors, 
pages 185--230, Springer 2006.

\bibitem[FG]{fg}Farber, M. Grant, M. Topological complexity of configuration spaces. Proc. Amer. Math. Soc. 137 (2009), 1841--1847.

\bibitem[FGY]{fgy}Farber, M. Grant, M. Yuzvinsky, S. Topological complexity of collision free motion planning algorithms in the presence of multiple moving obstacles, "Topology and Robotics” (M. Farber, R. Ghrist et al editors), Contemporary Mathematics AMS, vol. 438, 2007, 75--83.

\bibitem[FY]{fy}Farber, M. Yuzvinsky, S. Topological robotics: subspace arrangements and collision free motion planning. Geometry, topology, and mathematical physics, 145--156, Amer. Math. Soc. Transl. Ser. 2, 212, Amer. Math. Soc., Providence, RI, 2004.


\bibitem[GG]{gg} Gonzalez, J.  Grant, M. Sequential motion planning of non-colliding particles in Euclidean spaces. 
To appear in Proc. Amer. Math. Soc.


\bibitem[Roth]{roth} Roth, F. On the Category of Euclidean Configuration Spaces and associated Fibrations. Geometry \& Topology Monographs 13 (2008) 447--461.

\end{thebibliography}
\end{document}